\documentclass[11pt,a4paper]{amsart}
\usepackage {amsmath, amssymb, a4wide, epsfig, enumerate, psfrag}
\usepackage[latin1]{inputenc}
\usepackage[english]{babel}
\usepackage[active]{srcltx}

\date{\today}

\author{Bertrand Deroin \and Nicolas Tholozan}
\thanks{B.D.'s research was partially supported by ANR-08-JCJC-0130-01,  ANR-09-BLAN-0116.}
\address{CNRS \\ UMR 8553 \\ D\'epartement de Math\'ematiques et Applications  \\ \'Ecole Normale Sup\'erieure \\ 45, rue d'Ulm \\ 75005 Paris, France.}
\email{bertrand.deroin@ens.fr}
\address{University of Luxembourg \\
Campus Kirchberg \\
Mathematics Research Unit \\
6, rue rue Richard Coudenhove-Kalergi \\
L-1359 Luxembourg}
\email{nicolas.tholozan@uni.lu}

\title{Dominating surface group representations by \mbox{Fuchsian ones}}

\newcommand{\R}{\mathbb{R}}

\newcommand{\norm}[1]{\left\Vert#1\right\Vert}

\newcommand{\Teich}{\mathcal{T}}

\newcommand{\C}{\mathbb{C}}

\renewcommand{\H}{\mathbb{H}}

\newcommand{\Isom}{\mathrm{Isom}}

\newcommand{\Tr}{\mathrm{Tr}}

\renewcommand{\d}{\mathrm{d}}

\renewcommand{\epsilon}{\varepsilon}

\newcommand{\PSL}{\mathrm{PSL}}

\newcommand{\PSO}{\mathrm{PSO}}

\renewcommand{\phi}{\varphi}
\newcommand{\PSU}{\mathrm{PSU}}

\DeclareMathOperator{\codim}{codim}

\newcommand{\Vol}{\mathrm{Vol}}

\renewcommand{\tilde}[1]{\widetilde{#1}}

\newtheorem{prop} {Proposition} [section]

\newtheorem{defi}[prop] {Definition}
\newtheorem{lem}[prop] {Lemma}

\newtheorem{theo}{Theorem}

\newtheorem{coro}[theo]{Corollary}

\theoremstyle{theorem}
\newtheorem*{CiteThm}{Theorem}
\theoremstyle{remark}

\newtheorem{rmk}[prop]{Remark}

\newtheoremstyle{nosThm}
{0.3cm} 
{0.3cm} 
{\itshape} 
{} 
{\scshape} 
{.} 
{\newline} 
{} 

\theoremstyle{nosThm}

\begin{document}

\begin{abstract} We prove that a representation from the fundamental group of a closed surface of negative Euler characteristic with values in the isometry group of a Riemannian manifold of sectional curvature bounded by $-1$ can be dominated by a Fuchsian representation. Moreover, we prove that the domination can be made strict, unless the representation is discrete and faithful in restriction to an invariant totally geodesic $2$-plane of curvature $-1$. When applied to representations into $\PSL(2,\R)$ of non-extremal Euler class, our result is a step forward in understanding the space of closed anti-de Sitter $3$-manifolds. \end{abstract} 

\maketitle 

\section*{Introduction}

Let $\Gamma$ be a group, and $\rho_i : \Gamma \rightarrow \text{Isom} (M_i, g_i) $ be representations of $\Gamma$ in the groups of isometries of Riemannian manifolds $(M_i, g_i)$, for $i = 1,2$. We will say that $\rho_1$ \textit{dominates} $\rho_2$ if there exists a $1$-Lipschitz map $ f : M_1 \rightarrow M_2 $ which is $(\rho_1, \rho_2)$-equivariant, i.e. which satisfies  
\begin{equation} \label{eq:equivariance}  f( \rho_1(\gamma) \cdot x) = \rho_2(\gamma) \cdot f(x) \end{equation}
for every $x\in M_1$ and $\gamma \in \Gamma$. The domination will be called \textit{strict} if the map $f$ can be chosen to be $\lambda$-Lipschitz for some constant $0<\lambda<1$. 

We will be mainly interested in the case where $\Gamma$ is the fundamental group of a closed oriented connected surface of negative Euler characteristic. Those surfaces are endowed with metrics of constant negative curvature $-1$, and any such metric gives rise to an isometric identification of $S$ with a quotient $j(\Gamma) \backslash \H^2$, where $\H^2$ is the Poincar\'e half-plane and $j$ a representation of $\Gamma$ into $\Isom^+(\H^2)\simeq \PSL(2,\R)$. The representations obtained by this procedure are called \emph{Fuchsian}. 

 We are interested in the following question: given a representation $\rho$ of the fundamental group of a closed oriented surface of negative Euler characteristic into the group of isometries of a complete Riemannian manifold, is there a Fuchsian representation of $\Gamma$ that dominates $\rho$? When $\rho$ takes values into $\PSL(2,\R)$, the question has been raised by Kassel in her work on closed anti-de Sitter $3$-manifolds (cf. subsection \ref{ss:AdS}).

Here we prove:

\begin{theo}\label{t:main thm}
Let $S$ be a closed oriented surface of negative Euler characteristic, $\Gamma$ its fundamental group, $(M,g)$ a smooth, simply connected, complete Riemannian manifold and $\rho : \Gamma \rightarrow \Isom(M,g)$ a representation. Assume that the sectional curvature of $(M,g)$ is bounded above by $-1$. Then there exists a Fuchsian representation $j$ of $\Gamma$ that dominates $\rho$. Moreover, the domination can be made strict unless $\rho$ stabilizes a totally geodesic plane $\H \subset M$ of curvature $-1$, in restriction to which $\rho$ is Fuchsian. 
\end{theo}

Note that the strict domination cannot occur if $\rho$ is Fuchsian in restriction to an invariant copy of $\mathbb H^2$. Indeed, composing the equivariant map $f$ with the orthogonal projection on $\H$ would provide a contracting map from the original hyperbolic surface to the new hyperbolic surface $\rho(\Gamma) \backslash \mathbb H$. This would contradict the fact that both surfaces have the same volume $-2\pi \chi(S)$.

One can think of several generalizations of Theorem \ref{t:main thm}. In section \ref{s:orbifold}, we will extend it to representations of lattices in $\text{PSL}(2,\mathbb R)$ with torsion. 
It is likely that our approach can be generalized to the case where $M$ is any $\text{CAT}(-1)$ space, see subsection \ref{ss: strategy}. 

Hereafter we discuss some applications of our result, and in the last section some possible developments to higher rank representations.

\subsection{Contraction of the length spectrum and universality of the Bers constant} 
Recall that the \textit{translation length} of an isometry $\varphi$ of a Riemannian manifold $(M,g)$ is defined by 
\begin{equation} \label{eq:translation length}  l(\phi) := \inf _{x\in M} d(x,\phi(x))~, \end{equation}
where $d(\cdot, \cdot) $ is the distance induced by $g$. If $\rho : \Gamma \rightarrow \text{Isom} (M,g)$ is a representation, the \textit{length spectrum} of $\rho $ is the map $L_\rho : \gamma \in \Gamma \mapsto l (\rho (\gamma) ) \in \mathbb R^+$. 

The length spectrum of a representation contains a lot of information, and sometimes determines completely the representation. For instance, Zariski dense representations in $\mathrm{PSO}(n,1)$ are determined by their length spectrum. Another illustration  of this phenomenon is a famous result of Otal \cite{Otal} stating that one can recover a negatively curved metric on a closed surface by the knowledge of the length spectrum of the action of its fundamental group on the universal cover. 

Given two representations $\rho_i : \Gamma \rightarrow \text{Isom} (M_i, g_i)$, we say that the length spectrum of $\rho_1$ dominates the one of $\rho_2$ if $L _ {\rho _2} \leq L_{ \rho_1}$. Similarly, the domination is called \textit{strict} if there is a positive constant $\lambda <1$ such that $L_{\rho_2} \leq \lambda L_{\rho_1}$. 

It is clear that if $\rho_1$ (strictly) dominates $\rho_2$, then the length spectrum of $\rho_1$ (strictly) dominates the length spectrum of $\rho_2$. The converse has been proven in several cases by Gu\'eritaud--Kassel \cite{GK}. They pointed out to us that their method would work in our setting. The following corollary can therefore be seen as a reformulation of Theorem \ref{t:main thm}.

\begin{coro} Let $\Gamma$ be the fundamental group of a closed oriented surface of negative Euler characteristic, and $\rho$ a representation of $\Gamma$ into the isometry group of a smooth simply connected complete Riemannian manifold of curvature bounded above by $-1$. Then there exists a Fuchsian representation of $\Gamma$ whose length spectrum dominates $L_\rho$. Moreover, the domination can be made strict unless there exists a totally geodesic copy of $\mathbb H^2$ preserved by $\rho$, in restriction to which $\rho$ is Fuchsian. 
\end{coro}

Recall that the Bers constant in genus $g$ is the smallest constant $B_g$ such that for any hyperbolic metric on a closed surface $S$ of genus $g$, there is a decomposition of $S$ into pairs of pants such that all the geodesics of this decomposition have length at most $B_g$. The Bers constant in genus $2$ has been explicitly computed by Gendulphe \cite{Gendulphe}.

Since the length spectrum of a representation of a surface group into the isometry group of any simply connected Riemannian manifold of sectional curvature $\leq -1$ can be dominated by the length spectrum of a Fuchsian representation, the Bers constant naturally extends to those representations, and we get the following corollary:

\begin{coro}
Let $\Gamma$ be the fundamental group of a closed oriented surface $S$ of genus $g\geq 2$, $M$ a smooth, simply connected, complete Riemannian manifold of curvature bounded above by $-1$ and $\rho$ a representation of $\Gamma$ into $\Isom(M)$. Then there exists a pants decomposition of $S$ for which the image of any curve of the decomposition has translation length  at most $B_g$.
\end{coro}

March\'e and Wolff \cite{MW} recently used this extension of the Bers constant to solve a conjecture of Bowditch in genus $2$. They proved that, given a closed oriented surface $S$ of genus $2$ and a non-Fuchsian representation $\rho: \pi_1(S) \to \PSL(2,\R)$, there always exists a simple closed curve in $S$ whose image by $\rho$ is not hyperbolic.

\subsection{Rigidity results}
Theorem \ref{t:main thm} asserts that Fuchsian representations are ``maximal'' in some strong sense that implies various weaker rigidity results.  To illustrate this, let us recall the definition of the critical exponent of a representation.

\begin{defi}
Let $\Gamma$ be a surface group, $(M,d)$ a metric space and $\rho: \Gamma \to \Isom(M)$ a representation.
The \emph{critical exponent} of $\rho$ is the smallest number $\delta(\rho)$ such that the Poincar\'e series
\[ \sum_{\gamma \in \Gamma} e^{-s\, d(x,\gamma\cdot x)}\]
converges for all $s> \delta(\rho)$. (The convergence does not depend on the choice of the base point $x \in M$).
\end{defi}

When $M$ is a $\mathrm{CAT}(-1)$ space and $\rho$ is convex cocompact, the critical exponent coincides with the Hausdorff dimension of the limit set of $\rho(\Gamma)$ in $\partial_\infty M$ (see for instance \cite{Coornaert}). It is equal to $1$ when $\rho$ is Fuchsian in restriction to a totally geodesic hyperbolic plane of curvature $-1$, and it easily follows from Theorem \ref{t:main thm} that it is greater than $1$ otherwise. As a corollary, we obtain a new proof of the following result:

\begin{coro}
Let $\Gamma$ be a surface group and $\rho$ a convex cocompact representation of $\Gamma$ into the isometry group of some complete simply connected Riemannian space $M$ of sectional curvature $\leq -1$. Then the limit set of $\rho(\Gamma)$ in $\partial_\infty M$ has Hausdorff dimension $\geq 1$, with equality if and only if $\rho$ is Fuchsian in restriction to some stable totally geodesic plane of curvature $-1$.
\end{coro}

When $M$ is the hyperbolic space $\H^3$, this is a famous theorem of Bowen \cite{Bowen}. It was conjectured by Bourdon \cite{Bourdon} and proved by Bonk and Kleiner \cite{BK} when $M$ is any $\mathrm{CAT}(-1)$ space.  A generalization of our approach in the context of general $\text{CAT}(-1)$-spaces, which seems likely (see subsection \ref{ss: strategy}), would then provide an alternative proof of Bourdon's conjecture. \\

Toledo introduced in \cite{Toledo} another notion of maximality for representations of a surface group into $\PSU(n,1)$. Recall that the symmetric space of $\PSU(n,1)$ is the \emph{complex hyperbolic space} $\H_\C^n$. It carries a $\PSU(n,1)$-invariant K\"ahler metric that can be normalized to have sectional curvature between $-4$ and $-1$. Our theorem thus applies for representations into $\PSU(n,1)$.

Let $S$ be a closed oriented surface of genus $g \geq 2$ and $\Gamma$ its fundamental group. Let $\rho$ be a representation of $\Gamma$ into $\PSU(n,1)$, and $f$ be any $(\Gamma,\rho)$-equivariant map from $\tilde{S}$ to $\H^n_{\C}$. Then the number
$$\tau (\rho) = \frac{2}{\pi} \int_S f^* \omega$$
is an integer independent of $f$, called the \emph{Toledo invariant}. Toledo proved \cite{Toledo, Toledo'} that this invariant is bounded between $2-2g$ and $2g-2$ and is extremal if and only if the representation is Fuchsian in restriction to some stable totally geodesic hyperbolic plane in $\H_\C^n$. Representations with extremal Toledo invariant are called \emph{maximal}.

Surprisingly, maximal representations in that sense do not have maximal length spectrum. Indeed, Toledo's maximal representations are Fuchsian in restriction to some holomorphic copy of $\H^2$, which has curvature $-4$. In constrast, representations that cannot be strictly dominated preserve a copy of $\H^2$ of curvature $-1$ which is Lagrangian with respect to $\omega$. Therefore, those representations have a Toledo invariant equal to $0$. We are thankful to Pr. Toledo for bringing this subtlety to our attention.

\subsection{Closed anti-de Sitter  manifolds of dimension 3} \label{ss:AdS}
Anti-de Sitter (AdS) manifolds are Lorentz manifolds of constant negative sectional curvature. In dimension 3, they are locally modelled on $\PSL(2,\R)$ equipped with its Killing metric, whose isometry group is (up to finite index) $\PSL(2,\R) \times \PSL(2,\R)$ acting on $\PSL(2,\R)$ by
$$(\gamma_1, \gamma_2) \cdot x = \gamma_1 x \gamma_2^{-1}~.$$

Klingler \cite{Klingler}, generalizing a result of Carri\`ere \cite{Carriere}, proved that closed Lorentz manifolds of constant curvature are always geodesically complete. A consequence is that a closed AdS manifold of dimension 3 is a quotient of the universal cover $\tilde{\PSL(2,\R)}$ by a subgroup of $\tilde{\PSL(2,\R)} \times \tilde{\PSL(2,\R)}$ acting freely, properly discontinuously and cocompactly on $\tilde{\PSL(2,\R)}$.

Those quotients have been described by works of Goldman \cite{Goldman}, Kulkarni and Raymond \cite{KR}, Salein \cite{Salein}, and Kassel \cite{Kassel'}. Kulkarni and Raymond proved that, up to a finite cover and a finite quotient, closed anti-de Sitter manifolds are isometric to $\Gamma_{j, \rho} \backslash \PSL(2,\R)$, where $\Gamma$ is a surface group, $j, \rho$ two representations of $\Gamma$ into $\PSL(2,\R)$, $j$ Fuchsian, and $\Gamma_{j,\rho}$ is the image of $\Gamma$ into $\PSL(2,\R) \times \PSL(2,\R)$ by the embedding
$$\gamma \mapsto (j(\gamma), \rho(\gamma))~.$$

A pair $(j,\rho)$ of representations such that $\Gamma_{j,\rho}$ acts properly discontinuously on $\PSL(2,\R)$ is called an \emph{admissible pair}. Salein noticed that a sufficient condition for $(j,\rho)$ to be admissible is that $j$ strictly dominates $\rho$.  As a consequence, he obtains the existence of admissible pairs $(j,\rho)$ with $\rho$ of any non-extremal Euler class. Lastly, Kassel \cite[Chapter 5]{Kassel'} proved that Salein's sufficient condition is also necessary. 

Describing the space of closed anti-de Sitter 3-manifolds (up to finite coverings) thus reduces to describing the set of triples $(\Gamma, j, \rho)$, where $\Gamma$ is a surface group, $j$ a Fuchsian representation of $\Gamma$ into $\PSL(2,\R)$ and $\rho$ another representation that is strictly dominated by $j$. A natural question is whether any non-Fuchsian representation $\rho$ can appear in an admissible pair. It is answered positively by Theorem \ref{t:main thm}. We thus get the following corollary:

\begin{coro}
Let $S$ be a closed oriented surface of negative Euler characteristic, $\Gamma$ its fundamental group, and $\rho : \Gamma \to \PSL(2,\R)$ a representation of non-extremal Euler class. Then there exists a Fuchsian representation $j$ of $\Gamma$ such that $\Gamma_{j,\rho}$ acts properly discontinuously and cocompactly on $\PSL(2,\R)$.
\end{coro}

This result has been obtained independently and with different methods by Gu\'eritaud, Kassel and Wolff, see \cite{GKW}.

\begin{rmk}
Theorems of Kulkarni--Raymond and Kassel have been generalized by Kassel \cite{Kassel} and Gu\'eritaud--Kassel \cite{GK} to compact quotients of $\PSO(n,1)$ by discrete subgroups of $\PSO(n,1) \times \PSO(n,1)$. Namely, they proved that those quotients have (up to finite index) the form
$$\Gamma_{j,\rho} \backslash \PSO(n,1)$$ 
with $\Gamma$ the fundamental group of a closed hyperbolic $n$-manifold and $j, \rho$ two representations of $\Gamma$ into $\PSO(n,1)$, $j$ discrete and faithful and $\rho$ strictly dominated by $j$. However, when $n\geq 3$, the picture is very different. Indeed, by Mostow's rigidity theorem, there is only one discrete and faithful representation of $\Gamma$ up to conjugacy. Therefore, one cannot change the translation lengths of the elements $j(\gamma)$, and when $H^1(\Gamma, \mathbb R) \neq 0$ it is possible to construct abelian representations $\rho$ that are not dominated by $j$.\\
\end{rmk}

\subsection{Strategy of the proof} \label{ss: strategy}
Our approach shares some similarities with the technique used by Toledo in  \cite{Toledo, Toledo'}. Starting with any hyperbolic metric on $S$, one can consider a $(\Gamma, \rho)$-equivariant map from $\tilde{S}$ to $M$ that is harmonic. (It almost always exists, according to a theorem of Labourie \cite{Labourie}.) Toledo noticed that such a map may not be $1$-Lipschitz, though it contracts volume in average. 

However, it turns out that this harmonic map becomes contracting after suitably changing the hyperbolic metric on $S$. If the harmonic map were an immersion, this could be achieved by uniformizing the metric of $M$ pulled back by the harmonic map. Domination would then follow from Lemma \ref{l:curvature bound} and the well-known Ahlfors--Schwarz--Pick lemma \cite{Ahlfors}.

In our setting, harmonic maps need not be immersions, but our strategy is similar. The new hyperbolic metric on $S$ is constructed using a uniformization theorem due independently to Hitchin \cite{Hitchin} and Wolf \cite{Wolf} (see subsection \ref{ss:Wolf}), and the Ahlfors--Schwarz--Pick lemma is replaced by a maximum principle relying on classical Bochner-type identities for harmonic maps, refining an argument of Sampson, see the proof of \cite[Theorem 13]{Sampson}.

Note that the representation that we obtain this way depends on the choice of the initial hyperbolic metric on $S$. What we construct is actually a map from the Teichm\"uller space of $S$ to the domain of Fuchsian representations strictly dominating $\rho$. In \cite{Tholozan}, the second author proves that this map is a homeomorphism, leading to a topological description of the space of pairs $(j,\rho)$ such that $j$ strictly dominates $\rho$.

It is likely that our approach can be generalized when $M$ is any $\text{CAT}(-1)$-space, using Korevaar--Schoen's analysis of harmonic mappings with non positively curved metric target space (see \cite{KS}), and particularly their definition of the Hopf differential in this context. However, some technical difficulties arise, because for a singular target space one cannot use directly Bochner identities.\\

The next section introduces the main results we need about harmonic maps, and we prove Theorem \ref{t:main thm} in section \ref{s:proof}. In section \ref{s:orbifold} we note that our proof naturally extends to representations of lattices of $\PSL(2,\R)$ that are not necessarily torsion-free.

\subsection{Acknowledgments}

We thank Fran\c cois Gu\'eritaud, Fanny Kassel and Maxime Wolff for interesting discussions about our respective points of view and Olivier Guichard, Michael Wolf and Domingo Toledo for useful comments that improved a first version of this work. We also thank Sorin Dumitrescu for his interest in this subject and his important role in the collaboration of the authors. The first author thanks the university of Nice-Sophia Antipolis for the invitation and the nice working conditions. 

\section{Harmonic maps between Riemannian manifolds}

A harmonic map between two Riemannian manifolds is a critical point of the \emph{energy functional}, which associates to a map $f$ the mean value of the square norm of its differential. 
Though most of the fundamental results we state here are true in a general setting, we will restrict ourselves to the theory of harmonic maps with the source being a surface. In this case, the energy of a map only depends on the conformal class of the metric on the source. More details about the theory of harmonic maps and its relation to Teichm\"uller theory can be found in the survey of Daskalopoulos and Wentworth \cite{DW}.

\subsection{Energy of a map and harmonicity}
Let $(S,g_0)$ be a Riemann surface, $(M,g)$ a Riemannian manifold, and $f: S \to M$ a smooth map. One can measure how ``stretchy'' the map $f$ is by comparing $f^*g$ with $g_0$. 

\begin{defi}
Let $g_1$ be a non-negative symmetric $2$-form on $S$. Let $A$ be the field of symmetric endomorphisms of $(TS, g_0)$ such that $g_1 = g_0(\cdot, A \cdot)$. The \emph{energy density} of $g_1$ (with respect to the Riemann structure $g_0$) is the function
$$x \mapsto e_{g_0}(g_1)(x) = \frac{1}{2}\Tr(A_x)~.$$
Let $f$ be a map from $(S,g_0)$ to a Riemannian manifold $(M,g)$. The \emph{energy density} of $f$ is the function
$$x \mapsto e_{g_0}(f)(x) = e_{g_0}(f^*g)(x)~.$$
\end{defi}
When there is no confusion on the ambient Riemann structure $g_0$, we will usually omit to index the energy density on $g_0$ and simply write ``$e(f)$''.

\begin{defi}
The \emph{total energy} of $f$ is the integral of the energy density:
$$E(f) = \int_S e(f)(x) \Vol_0(x)~,$$
where $\Vol_0$ is the volume form induced on $S$ by the metric $g_0$.
\end{defi}

Consider now a representation $\rho: \Gamma = \pi_1(S)\to \Isom(M,g)$, and $f: \tilde{S}\to M$ a $(\Gamma,\rho)$-equivariant map. Since $\rho(\Gamma)$ acts on $M$ by isometries, the symmetric $2$-form $f^*g$ on $\tilde{S}$ is preserved by the action of the fundamental group, and so is the energy density $e_{\tilde{g_0}}(f)$. We will denote $e_{g_0}(f)$ the induced function on $S$. Then we similarly call the integral of $e_{g_0}(f)$ against $\Vol_0$ on $S$ the total energy of $f$ and denote it $E(f)$.\\

From now on we assume $S$ is closed. Suppose $f$ is a map from $S$ to $M$ that minimizes the energy functional among all smooth maps homotopic to $f$. Then $f$ must verify a certain partial differential equation that can be expressed as the vanishing of a differential operator called the \textit{tension field} of the map. 

\begin{defi}
The \emph{second fundamental form} of $f$ is the section of $\text{Sym}^2 T^* S \otimes f^* TM $  defined by
\begin{equation}\label{eq:second fundamental form of f} II ^f (X,Y) := \nabla^f _{ X} df( Y)  - df ( \nabla^{g_0} _X Y), \end{equation}
where $\nabla^f$ is the pull-back on $f^* TM$ of the Levi-Civita connexion $\nabla ^g$. 

The \emph{tension field} of $f$ is  
\begin{equation} \tau (f) := \Tr_{g_0}\ II^f = II ^f (e_1,e_1) + II^f (e_2,e_2)  \end{equation}
where $(e_1,e_2)$ is any orthonormal basis of $(TS, g_0)$. 

The map $f$ is \emph{harmonic} if the tension field $\tau(f)$ vanishes everywhere.
\end{defi}

\begin{prop}
If $f: S \to M$ minimizes the energy functional among all maps homotopic to $f$, then $f$ is harmonic.
\end{prop}

For the same reason, if a $(\Gamma,\rho)$-equivariant map from $\tilde{S}$ to $M$ minimizes the energy functional among all $(\Gamma,\rho)$-equivariant maps, then it is harmonic.
Note that the total energy of $f$ depends on the metric $g$ on $M$, but only on the conformal class of $g_0$ on $S$. Indeed, if we multiply $g_0$ by some positive function $\sigma$, the energy density of $f$ is divided by $\sigma$, but the volume is multiplied by $\sigma$, and therefore the total energy is preserved. From this, one can deduce that harmonicity only depends on the conformal class of $g_0$. This is specific to the case where $S$ has dimension $2$.

\subsection{Existence results}

Existence results for harmonic maps date back to the seminal work of Eells--Sampson \cite{ES} where they deal with maps between closed manifolds. Here we will need an analogous result for equivariant maps.

Recall that a simply connected Riemannian manifold $M$ with sectional curvature $\leq -1$ is \emph{Gromov hyperbolic}, see \cite{CDP, GD}. One can thus define its \emph{boundary} $\partial_\infty M$ as the space of geodesic rays, where two such rays are identified when they remain at bounded distance. Any isometry of $M$ induces a transformation of $\partial_\infty M$. Therefore, a representation $\rho : \Gamma \to \Isom(M)$ induces an action of $\Gamma$ on $\partial_\infty M$.

\begin{CiteThm}[Labourie]
Let $S$ be a closed  Riemann surface, $\Gamma$ its fundamental group, $(M,g)$ a complete simply connected Riemannian manifold of sectional curvature $\leq -1$, and $\rho$ a representation of $\Gamma$ into $\Isom (M,g)$. If $\rho$ does not fix a point in the boundary of $M$, or if $\rho$ fixes a geodesic in $M$, then there exists a $\rho$-equivariant harmonic map from $\tilde{S}$ to $M$. If $\rho$ does not fix a point in the boundary, this map is unique.\\
\end{CiteThm}

\begin{rmk}
This statement is only a particular case of Labourie's theorem, that deals with spaces of non-negative curvature (for which the condition on $\rho$ is more difficult to express). Labourie's result was first obtained by Donaldson for maps into $\H^3$ \cite{Donaldson} and by Corlette for maps into non-positively curved Riemannian symmetric spaces \cite{Corlette}.
\end{rmk}

\subsection{Hopf differential and parametrization of Teichm\"uller space}
\subsubsection{Harmonic maps and quadratic differentials}
Consider a surface $S$ with a Riemannian metric $g_0$. This metric induces a conformal structure on $S$ and thus a complex structure. Any complex symmetric $2$-form on $S$ splits into a $(1,1)$-part, a $(2,0)$-part and a $(0,2)$-part. In particular, if $f: S\rightarrow (M,g)$ is a smooth map, $f^*g$ can be written in the form $\alpha g_0 + \Phi + \bar{\Phi}$, where $\Phi$ is a quadratic differential called the \emph{Hopf differential}. The following proposition is classical. A proof can be found in \cite[Section 2.2.3]{DW}. 

\begin{prop}\label{p:quadratic differential}
If $f: (S, [g_0]) \to (M,g)$ is harmonic, then its Hopf differential is holomorphic. This necessary condition is also sufficient when $M$ is a surface.
\end{prop}

Observe that if $f : \tilde{S} \rightarrow (M,g) $ is an equivariant harmonic map with respect to some representation $\rho : \pi_1(S) \rightarrow \text{Isom} (M,g)$, then the Hopf differential of $f$ on $\tilde{S}$ is invariant under the action of the fundamental group of $S$, and thus defines a holomorphic quadratic differential on $S$.

\subsubsection{Harmonic maps and Teichm\"uller space} \label{ss:Wolf}

Let $S$ be a closed surface of negative Euler characteristic. We view here the Teichm\"uller space $\Teich(S)$ as the space of hyperbolic metrics on $S$, where two such metrics are identified when there is an isometry between them which is homotopic to the identity on $S$. Fixing a base point $g_0$ in Teichm\"uller space, and another point $g_1$, one can look at the unique harmonic map from $(S,g_0)$ to $(S,g_1)$ homotopic to the identity. The Schoen--Yau theorem \cite{SY} states that this map is a diffeomorphism. By Proposition \ref{p:quadratic differential}, the Hopf differential of this map is holomorphic with respect to the complex structure given by $g_0$. This constructs a map from the Teichm\"uller space to the vector space of holomorphic quadratic differentials on $(S,[g_0])$. Sampson \cite{Sampson} proved that this map is injective and Hitchin \cite{Hitchin} and Wolf \cite{Wolf} proved that it is surjective. Those results can be summed up in one theorem.

\begin{CiteThm}[Schoen--Yau, Sampson, Hitchin, Wolf]
Let $(S, [g_0])$ be a closed Riemann surface, and $\Phi$ a holomorphic quadratic differential on $S$. Then there is a unique Riemannian metric $g_1$ on $S$ of curvature $-1$ such that
$$g_1 = \alpha g_0 + \Phi + \bar{\Phi}$$
for some positive function $\alpha$.\\
\end{CiteThm}

\section{Proof of Theorem \ref{t:main thm}} \label{s:proof}

\subsection{Representations fixing a point at infinity} \label{ss:ElementaryCase}
Our strategy for proving that any representation $\rho: \pi_1(S)  \rightarrow \text{Isom} (M,g)$  can be dominated is based on the existence of a $\rho$-equivariant harmonic map from $\tilde{S}$ to $M$. Hence, we must first say something about representations fixing a point in $\partial_\infty M$, for which Labourie's theorem does not hold. Fortunately, in this case, there is a trick which reduces the problem to the (easy) abelian case.

Assume that $\rho(\pi_1(S))$ fixes a point $a\in \partial_\infty M$. Given a geodesic ray $\gamma: [0,\infty) \rightarrow M$ which tends to $a$ at infinity, we classically define the Busemann function as
\begin{equation} \beta_\gamma (x) :=  \lim_{t\rightarrow \infty} \left(  d(\gamma(t), x) -t \right). \end{equation} 
Under the assumption that $(M,g)$ is simply connected with curvature bounded away from $0$, the following properties hold: for every element $\phi\in \text{Isom} (M,g)$ such that $\phi(a) = a$, there exists a real number $m (\phi)$ such that for every $x\in M$,
\begin{equation}\label{eq:morphism} \beta_\gamma ( \phi(x)) = \beta_\gamma (x) + m(\phi) ,\end{equation}
and moreover 
\begin{equation}\label{eq:length vs busemann} l (\phi) = |m(\phi)|,\end{equation}
where we recall that $l(\phi)$ is the translation length defined by \eqref{eq:translation length}. 
We refer to the books \cite{CDP, GD} for details about the theory of Gromov hyperbolic spaces. 

In the case where $\rho: \pi_1(S) \rightarrow \text{Isom} (M,g)$ fixes the point $a\in \partial_\infty M$, the function {$m\circ \rho : \pi_1(S)\rightarrow \mathbb R$} is a morphism by \eqref{eq:morphism}. Let $\gamma'$ be an oriented bi-infinite geodesic in the hyperbolic plane $\H^2$ of constant curvature $-1$ and let $\rho'$ be the representation from $\pi_1(S)$ to $\H^2$ preserving the geodesic $\gamma'$ and acting on $\gamma'$ by translations given by $m\circ \rho$. Equation \eqref{eq:length vs busemann} reduces the problem of dominating $\rho$ to the problem of dominating $\rho'$. But now, $\rho'$ fixes a geodesic, and Labourie's theorem can thus be applied to $\rho'$. We obtain a $(\Gamma,\rho')$-equivariant harmonic map taking values in the geodesic. This harmonic map can be constructed by integrating the harmonic $1$-form having the cohomology class of $m\circ \rho \in H^1(S,\mathbb R)$.

\subsection{Proof of Theorem \ref{t:main thm}}

Let $S$ be a closed oriented surface of negative Euler characteristic, $\Gamma$ its fundamental group, and $\rho$ a representation of $\Gamma$ into $\Isom(M,g)$ that does not fix a point in $\partial_\infty M$. Fix an arbitrary hyperbolic metric $g_0$ on $S$. By Labourie's theorem, we can consider a $(\Gamma,\rho)$-equivariant harmonic map $f : (\tilde{S}, \tilde{g_0}) \rightarrow (M,g)$.  Let $\Phi$ be the holomorphic quadratic differential on $S$ such that 
$$f^* g = \alpha g_0 + \Phi + \bar{\Phi}$$
for some real function $\alpha$.

Wolf's theorem (combined with Schoen--Yau) gives us a hyperbolic metric $g_1$ on $S$ such that $g_1 = \alpha_1 g_0 + \Phi + \bar{\Phi}$ for some positive function $\alpha_1$. Then our main theorem is the direct consequence of the following lemma:

\begin{lem} \label{Contracting}
Either $f$ induces a diffeomorphism from $\tilde{S}$ to a totally geodesic plane $\H^2 \subset M$ of curvature $-1$, or we have
$$f^* g < g_1$$ on all $S$. 
\end{lem}

Indeed, if we know this inequality holds, and since $S$ is compact, there is a constant $\lambda < 1$ such that $f^* g \leq \lambda^2 g_1$. Let $j_1$ be a holonomy representation of $g_1$ (i.e. a representation of $\Gamma$ into $\PSL(2,\R)$ such that $(S,g_1)$ is isometric to $j_1(\Gamma) \backslash \H^2$).
Then $f: (\tilde{S},\tilde{g_1}) \to \H^2$ is $\lambda$-Lipschitz and $(j_1, \rho)$-equivariant, and therefore $j_1$ dominates $\rho$.\\

\subsubsection{The functions $H_i$ and $L_i$}
Recall that $g_0$ induces a natural hermitian metric on the line bundle $K_S^2$. When given a quadratic differential $\Phi$, the function $|\Phi|^2_{g_0}$ is thus well-defined on $S$.

Let us fix a metric $g_0$ of curvature $-1$ on $S$, and let $g'$ be a non-negative symmetric $2$-form on $S$ of the form $$\alpha g_0 + \Phi + \bar{\Phi},$$ with $\alpha$ a real function and $\Phi$ a holomorphic quadratic differential. 

\begin{lem} \label{l:proposition}
We have:
$$e_{g_0}(g') = \alpha,$$
$$\det_{g_0} g' = e(g')^2 - 4 |\Phi|^2_{g_0}~.$$
\end{lem}

\begin{proof}
In a local complex coordinate $z$, we have
$$g_0 = \sigma \d z \d\bar{z}$$
for some real positive function $\sigma$, and
$$g' = \alpha \sigma \d z \d \bar{z} + \phi \d z^2 + \bar{\phi} \d \bar{z}^2$$
for some complex valued (holomorphic) function $\phi$.
In coordinates $(x=\Re(z),y= \Im(z))$, we thus get
$$g_0 = \sigma(\d x^2 + \d y^2)$$
and $$g' = (\alpha \sigma + 2 \Re(\phi)) \d x^2 + (\alpha \sigma - 2 \Re(\phi))\d y^2 + 4 \Im(\phi) \d x \d y.$$
From this, we deduce that
$$e_{g_0}(g') = \alpha$$
and
$$\det_{g_0}(g') = \alpha^2 - 4\frac{| \phi|^2}{\sigma^2} = \alpha^2 - |\Phi|^2_{g_0}.$$
\end{proof}

Since $g'$ is non-negative, we obtain that $e(g')^2 - 4 |\Phi|^2_{g_0}\geq 0$, from which we can deduce that the system of equations
$$
\left\{
\begin{array}{rcl}
x+y &=& e(g')\\
xy &=& |\Phi|^2_{g_0}
\end{array}
\right.
$$
has two non-negative (eventually identical) solutions. We will use the following lemma:

\begin{lem}
Let $H$ and $L$ be the two functions on $S$ such that 
\begin{itemize}
\item $H \geq L$
\item $H+L = e_{g_0}(g')$
\item $HL = |\Phi|^2_{g_0}$
\end{itemize}
Then $(H-L)^2 = \det_{g_0} g'$, and wherever $g'$ is non degenerate, $H$ and $L$ are solutions of the partial differential equation:
\begin{equation} \label{eq:Bochner}
\Delta_0 \log(u) = -2 \kappa(g') \left( u - \frac{|\Phi|^2}{u} \right) -2
\end{equation}
where $\kappa(g')$ denotes the Gauss curvature of the Riemannian metric $g'$.
\end{lem}

\begin{proof}
The fact that $(H-L)^2 = \det_{g_0} g'$ is just a reformulation of Lemma \ref{l:proposition}. The second point is a classical fact that can be found for instance in Schoen--Yau's paper \cite{SY} (see also \cite{Wolf}). Let us denote $U$ the domain of $S$ where $g'$ is non degenerate and $f$ the identity on $U$, seen as a map from $(U,g_0)$ to $(U,g')$. Then
\[H = \norm{\partial f}^2\]
and 
\[L = \norm{\bar{\partial} f}^2\]
(see \cite[Equation (6)]{SY}). Equation \eqref{eq:Bochner} then follows from the fact that $f$ is harmonic, since its Hopf differential is holomorphic \cite[equations (16) and (17)]{SY}.
\end{proof}

Let's go back to our setting, where $f^*g = e(f) g_0 + \Phi + \bar{\Phi}$ and where $g_1 = e(g_1) g_0 + \Phi + \bar{\Phi}$. We introduce $H_1$ and $L_1$ such that
$$H_1 \geq L_1,$$
$$e(g_1) = H_1+ L_1$$ and
$$H_1 L_1 = |\Phi |^2~.$$
Similarly, let $H_2$ and $L_2$ be such that $H_2 \geq L_2$,
$$e(f) = H_2 + L_2$$ and 
$$H_2 L_2 = |\Phi |^2~.$$
Then we have:
\begin{itemize}
\item $L_1, H_1, L_2, H_2$ are non-negative, and $H_1 > L_1$ everywhere,
\item $\Vol_1 = (H_1-L_1)\Vol_0$,
\item $\det_{g_0} f^*g = (H_2-L_2)^2$,
\item $H_1, L_1$ are both solutions of the following partial differential equation: 
\begin{equation} \label{pde}
\Delta_0 \log(u) = 2u - 2 \frac{|\Phi|^2}{u}-2
\end{equation}
where $\Delta_0$ is the Laplace operator associated to the metric $g_0$,
\item On the open set $U$ of $S$  where $f^*g$ is non degenerate, $H_2, L_2$ are both smooth and solutions of 
\begin{equation} \label{pde2}
\Delta_0 \log(u) = 2\beta u - 2 \beta \frac{|\Phi|^2}{u} - 2,
\end{equation}
where $\beta = -\kappa$ and $\kappa$ is the sectional curvature of the Riemannian metric $f^* g$. 
\end{itemize}

\begin{rmk}
The domain $U$ where $f^*g$ is non-degenerate is either empty or dense, by \cite[Corollary of Theorem 3]{Sampson}.
\end{rmk}

The following result is presumably well-known. For instance, it can be read between the lines in \cite[Lemma C.4]{Reznikov} and \cite[Theorem 7]{Sampson}. Since it is crucial in our approach, we detail the argument here.

\begin{lem} \label{l:curvature bound}
For all $x\in U$ we have $\kappa (f^* g)(x) \leq -1$. Moreover, the inequality is strict, unless the second fundamental form of $f(\tilde{S})$ vanishes at $f(\tilde{x})$ ($\tilde{x}$ being any lift of $x$ in $\tilde{S}$). In particular, if $\kappa (f^* g)$ is identically $-1$ on $U$, then the image of $\tilde{U}$ by $f$ is totally geodesic. 
\end{lem}

\begin{proof}
By definition of $U$, $f$ restricted to the lift $\tilde{U}$ is an immersion. Let $V\subset \tilde{U}$ be an open subset, small enough so that $N : = \tilde{f}(V)$ is an embedded submanifold. Since $f: (V, f^*g) \to (N,g_N)$ is an isometry, the only thing we want to prove is that $(N, g_N)$ has curvature $\leq -1$. Take an orthonormal frame $e_1, e_2$ of $TN$. The curvature of $N$ is related to the sectional curvature of $TN$ in the ambient space $M$ by the following relation:
\begin{equation}  \kappa ^N   = \kappa ^M (TN) + \langle II^N(e_1,e_1), II^N(e_2,e_2) \rangle - || II^N(e_1,e_2)|| ^2, \end{equation}
where $II^N (u,v) $ is the second fundamental form of $N$. This formula can be re-expressed as 
\begin{equation} \label{eq:extrinsic intrinsic curvatures} \kappa^N = \kappa^M (TN) + \frac{\codim(N)}{\text{Jac}(f)^2} \cdot \mathbb E \left( \det_{g_0} ( \langle II^N (df(\cdot), df(\cdot) ), n \rangle)  \right) \end{equation}
where the average is taken over all the unitary vectors $n$ normal to $N$ with respect to normalized Haar measure, and $\text{Jac}(f)$ stands for the Jacobian of the map $f : (V, g_0) \rightarrow (N, g)$. 

The second fundamental form of $f$ (see \eqref{eq:second fundamental form of f}) and of $N$ are related by  
\begin{equation} \label{eq:second fundamental form} 
II^f (X,Y)= II ^N ( df (X), df(Y)) + df \left( \nabla _X^{f^* g} Y - \nabla _X^{g_0} Y\right).
\end{equation}
In particular, since both summands on the right hand side are orthogonal, using the harmonicity of $f$, we infer that $$\Tr_{g_0} \ II^N (df(\cdot), df(\cdot))=0 ,$$ and in particular for every unitary vector $n\in TM$ normal to $TN$, we get 
$$  \Tr_{g_0} \langle II^N (df(\cdot), df(\cdot)), n \rangle =0. $$
This shows that the eigenvalues of the quadratic form $ \langle II^N (df(\cdot), df(\cdot)), n \rangle$ are opposite, hence $\det_{g_0} ( \langle II^N (df(\cdot), df(\cdot) ), n \rangle)\leq 0$, with equality if only if the quadratic form \mbox{$\langle II^N (df(\cdot), df(\cdot)), n \rangle$} vanishes. From \eqref{eq:extrinsic intrinsic curvatures}, we deduce that $\kappa^N\leq \kappa^M(TN) \leq -1$. Equality implies the vanishing of the second fundamental form. This proves the lemma.
\end{proof}

\begin{lem}\label{Ineq}
Either $f$ induces a diffeomorphism from $\tilde{S}$ to a totally geodesic plane $\mathbb H^2 \subset M$ of curvature $-1$, or we have $$H_2 <H_1.$$
\end{lem}

\begin{proof}
Recall that $U$ is the  domain of $S$ where $f^*g$ is non-degenerate. First, note that on the complement of $U$, we have
$$H_ 2 = L_ 2 = \sqrt{|\Phi|^2_{g_0}} = \sqrt{L_1 H_1} < H_ 1.$$ 
We shall then only focus on what happens in the domain $U$.

 Note that $H_1$ does not vanish since $H_1>L_1 \geq 0$. Let $x$ be a point on $S$ such that $H_2(x)/H_1(x)$ is maximal. Assume by contradiction that $H_2(x) > H_1(x)$. Then since $H_2 L_2 = H_1 L_1$, we have $L_2(x) < H_2(x)$, and therefore $x$ belongs to $U$. Equations \eqref{pde} and \eqref{pde2} and the relation $H_2 L_2 = |\Phi|^2_{g_0}$ show that at the point $x$, 
\begin{equation} \label{eq:long equation} \Delta_0 \log \left( \frac{H_2}{H_1} \right)  = 2 (H_2 - H_1 ) +  2 |\Phi|^2_{g_0} \left( \frac{1}{H_1} - \frac{1}{H_2} \right)  + 2 (\beta -1) \left( H_2 -   L_2 \right), \end{equation}
where $\beta = -\kappa(f^*g)$. By Lemma \ref{l:curvature bound}, we have $\beta \geq 1$. Since $H_2 \geq L_2$ by hypothesis, the last summand in equation \eqref{eq:long equation} is non negative. Therefore, the assumption $H_2 > H_1$ clearly implies that $\Delta_0 \log \left( \frac{H_2}{H_1} \right)(x) >0$, which contradicts the fact that $\log \left( \frac{H_2}{H_1} \right)$ admits a maximum at $x$. At this maximum, we must have $H_2 \leq H_1$, and thus $H_2 \leq H_1$ everywhere.

To get the strict bound, we will use the following version of the strict maximum principle, which was probably already known by Picard (see \cite{Minda}, Theorem 1 for a slightly more general version).

\begin{lem}[Picard] \label{StrictIneq}
Let $w$ be a real non-positive function on a domain $U$ of $\C$, such that
$\Delta w \geq K w$ for some constant $K>0$. Then either $w \equiv 0$ on $U$, or $w<0$ on $U$.
\end{lem}
In this theorem, $\Delta$ is a priori the Laplace operator associated to a flat metric, but, since it is a local result, the conclusion still holds for any conformal metric. We apply Lemma \ref{StrictIneq} to the function  $w = \log \left( \frac{H_2}{H_1}\right)$. Since $\beta \geq 1$, equation \eqref{eq:long equation} shows that $$\Delta_0 \log \left( \frac{H_2}{H_1} \right)  \geq \left( 2 + \frac{2|\Phi|^2_{g_0}}{H_1 H_2} \right) (H_2 - H_1 ),$$ so that we get 
\[   \Delta_0 w \geq 2(H_1 + L_2) ( e^w -1 ) \geq K w , \]
where $K := 2 \max_S (H_1+L_2)$. Lemma \ref{StrictIneq} shows that either $H_2$ is identically equal to $H_1$ on $U$, or $H_2 < H_1$ on $U$. But on the complement of $U$, we already saw that $H_2<H_1$. In the case $H_2 = H_1$, the complement of $U$ is thus empty and equation \eqref{eq:long equation} shows that necessarily $\beta = 1$ on $S$, hence $\kappa (f^* g) = -1$ on $S$. Lemma \ref{l:curvature bound} shows that $\mathbb H^2 = f(\tilde{S})$ is a totally geodesic plane of curvature $-1$. Moreover, in that case, we also have $L_2 = L_1$ and thus $f^*g = g_1$, which means that $f: (S,g_1) \to M$ is an isometric embedding. Hence Lemma \ref{Ineq} is proved. 
\end{proof}

Let us finish the proof of Lemma \ref{Contracting}.
According to Lemma \ref{Ineq}, if $f$ is not an isometric embedding with totally geodesic image, then we have $H_2 < H_1$ everywhere. Since $H_2 L_2 = H_1 L_1$ we also have $L_2 > L_1$. Therefore, $(H_2-L_2)^2 < (H_1-L_1)^2$, and by adding $4H_2 L_2 = 4 H_1 L_1$ to each member, we get that $(H_2 + L_2)^2 < (H_1+L_1)^2$. 

Now, remember that $f^*g = (H_2 + L_2) g_0 + \Phi + \bar{\Phi}$ and $g_1 = (H_1 + L_1) g_0 + \Phi + \bar{\Phi}$. It is then clear that $H_2 +L_2 < H_1 +L_1$ implies $f^*g < g_1$.

\section{Extension of the theorem to lattices in $\mathrm{PSL}(2,\R)$ with torsion} \label{s:orbifold}



Here we extend the theorem when $\Gamma$ is a lattice in $\text{PSL}(2,\R)$ with torsion.

\begin{theo}
Let $\Gamma$ be a lattice in $\PSL(2,\R)$ and $\rho$ a representation of $\Gamma$ into the isometry group of a smooth complete simply connected Riemannian manifold $M$ of sectional curvature $\leq -1$. Then there exists a Fuchsian representation $j: \Gamma \to \PSL(2,\R)$ and a $(j,\rho)$-equivariant map from $\H^2$ to $M$ which is either a contraction or an isometric and totally geodesic embedding.
\end{theo}

\begin{rmk}
Note that this result is relevant to the study of anti-de Sitter Lorentz manifolds of dimension $3$. Indeed, even when $\Gamma$ is a lattice in $\PSL(2,\R)$ with torsion, the action of $\Gamma_{j,\rho}$ on $\PSL(2,\R)$ may not have any fixed point. If $j$ strictly dominates $\rho$, then the action of $\Gamma_{j,\rho}$ on $\PSL(2,\R)$ is properly discontinuous and cocompact. It is free if, furthermore, for any $\gamma \in \Gamma$ with torsion, $\rho(\gamma)$ has order strictly smaller than $\gamma$. In that case, the quotient $\Gamma_{j,\rho} \backslash \PSL(2,\R)$ is a smooth anti-de Sitter $3$-manifold which is a Seifert bundle over the orbifold $j(\Gamma) \backslash \H^2$.
\end{rmk}

\begin{proof}

By Selberg's lemma, we can consider a finite index normal subgroup $\Gamma_0 \subset \Gamma$ which is torsion-free. Hence the quotient $S= \Gamma_0 \backslash \H^2$ is a closed hyperbolic surface. We now mimic the proof of the main theorem.

Assume first that the action of $\rho(\Gamma)$ on $\partial_\infty M$ does not have any finite orbit, $\rho(\Gamma_0)$ does not fix a point in $\partial_\infty M$. Let $f$ be the unique $(\Gamma_0, \rho)$-equivariant harmonic map from $\H^2$ to $M$. We prove that $f$ is actually $(\Gamma, \rho)$-equivariant. Indeed, for some $\gamma \in \Gamma$, consider the map $\gamma \cdot f: x \to \rho(\gamma)^{-1}f(\gamma \cdot x)$. It is harmonic and $(\Gamma_0, \rho)$-equivariant. By uniqueness, $\gamma \cdot f = f$, and we get that $f(\gamma \cdot x) = \rho(\gamma) \cdot f(x)$. This being true for any $\gamma \in \Gamma$, we obtain that $\tilde{f}$ is $(\Gamma,\rho)$-equivariant.

Let $g_0$ denote the hyperbolic metric on $S=\Gamma_0 \backslash \H^2$. Then $g_0$ is $\Gamma / \Gamma_0$-invariant, and so is the Hopf differential $\Phi$ of $\tilde{f}$. Hence, the unique hyperbolic metric $g_1$ on $S$ of the form $g_1 = \alpha g_0 + \Phi + \bar{\Phi}$ is $\Gamma / \Gamma_0$-invariant. Therefore, $g_1$ induces an orbifold hyperbolic metric on the quotient of $S$ by $\Gamma / \Gamma_0$. The holonomy of this metric gives a representation $j: \Gamma \to \PSL(2,\R)$. By construction, the map $\tilde{f}$ is $(j,\rho)$-equivariant and for the same reason as before, it is either a contraction or an isometric and totally geodesic embedding.

Again, we must deal separately with the specific case where $\rho(\Gamma)$ has a finite orbit in $\partial_\infty M$. If this orbit contains at least three points, then it has a well-defined barycenter in $M$ which is fixed by $\rho(\Gamma)$ and the length spectrum of $\rho$ is identically $0$. When this orbit has one or two points, the length spectrum of $\rho$ is the same as the one of some representation $m: \Gamma \to \Isom(\R)$ (see section \ref{ss:ElementaryCase}).

In that case, one can still consider a $(\Gamma_0,m)$-equivariant harmonic map $f: \H^2 \to \R$. Some care must be taken because this map is only unique up to translation. The map $\gamma \cdot f$ defined as before is still $(\Gamma_0, m)$-equivariant and harmonic. It differs from $f$ by a translation, so that they both have the same Hopf differential $\Phi$. Therefore, $\Phi$ is $\Gamma$-invariant and the Fuchsian representation associated to $g_1$ still extends to a representation $j:\Gamma\to \PSL(2,\R)$. Lastly, the map $f$ may not be $(j,\rho)$-equivariant, but one easily checks that the barycenter map
\[\bar{f} = \frac{1}{[\Gamma:\Gamma_0]}\sum_{\gamma \in \Gamma/\Gamma_0} \gamma \cdot f\]
is $(j,\rho)$-equivariant and still contracting.

\end{proof}

\section{Perspectives for higher rank representations}

If $G$ is a simple Lie group of rank $\geq 2$ and $X$ its symmetric space, then the sectional curvature of $X$ is not bounded away from $0$ and our theorem does not apply to representations in $G$. We shall state some remarks and questions in this context.

First, we need to conveniently choose a normalization of the metric on $X$. To do this, one can note that, though the sectional curvature is not pinched away from $0$, there is a negative upper bound on the curvature of totally geodesic hyperbolic planes. We  choose to normalize the metric on $X$ so that this upper bound is exactly $-1$.

With this convention, it is not true in general that a representation of $\Gamma$ into $G$ can be dominated by a Fuchsian one unless it is Fuchsian in restriction to some stable hyperbolic plane of curvature $-1$. 
When $G$ is $\PSL(n,\R)$, for instance, those Fuchsian representations preserving a totally geodesic plane of curvature $-1$ can be continuously deformed into so-called \emph{Hitchin representations}. In contrast with Theorem \ref{t:main thm}, several results tend to show that Fuchsian representations are ``minimal'' among Hitchin representations (see \cite[Proposition 10.1]{Hitchin'}, and more recently \cite{PS}). 

In the case where $n=3$, the work of Loftin \cite{Loftin}, together with a recent result of Benoist--Hulin \cite[Proposition 3.4]{BH}, implies that one can find Hitchin representations ``as big as we want''. More precisely, for any Fuchsian representation $j$ and any constant $C$, there is a Hitchin representation $\rho$ in $\PSL(3,\R)$ such that \[L_\rho \geq C L_j\]
(see \cite{Nie} or \cite[Corollary 3.6]{Tholozan'}). We expect this to be true in any dimension $n\geq 3$.

Our theorem can thus be generalized as follows: for any linear representation $\rho$, there is a Hitchin representation $j$ in $\PSL(3,\R)$ such that $L_j \geq L_\rho$. However, this fact cannot give any interesting control on higher rank representations such as a systole or a Bers constant, precisely because Hitchin representations can be ``as big as we want''.\\

A more subtle generalization of our theorem could arise in the theory of Higgs bundles. Indeed, one could hope that, in any fiber of the Hitchin fibration, the Hitchin representation maximizes every translation length. For representations in $\PSL(2,\C) = \Isom^+(\H^3)$, this is a reformulation of our theorem. A generalization in higher dimension would be a step forward in understanding the relation between geometric properties of linear representations and their parametrization by Higgs bundles.

\end{document}